\documentclass[a4paper]{article}

\usepackage{amsthm,amsmath,amssymb}
\usepackage{proof}

\newtheorem{theorem}{Theorem}[section]
\newtheorem{lemma}[theorem]{Lemma}

\newtheorem{remark}[theorem]{Remark}

\newcommand{\T}{\mbox{\sf true}}
\newcommand{\F}{\mbox{\sf false}}

\newcommand{\IMP}{\mbox{$\to$}} 
 
\newcommand{\BOT}{\mbox{$\bot$}}

\newcommand{\GA}{{\mit \Gamma}}
\newcommand{\DE}{{\mit \Delta}}
\newcommand{\SI}{{\mit \Sigma}}
\newcommand{\PI}{{\mit \Pi}}
\newcommand{\PH}{{\mit \Phi}}
\newcommand{\PS}{{\mit \Psi}}

\newcommand{\TA}{{\mit \Theta}}

\newcommand{\SEQ}[2]{\mbox{${#1}\Rightarrow{#2}$}}
\newcommand{\SSEQ}[2]{\mbox{${#1}\Rrightarrow{#2}$}}
\newcommand{\SSSEQ}[2]{\mbox{${#1}\blacktriangleright{#2}$}}

\newcommand{\GL}{{\bf GL}}
\newcommand{\GLS}{{\bf GLS}}
\newcommand{\SF}[1]{{\rm Sub}(#1)}
\newcommand{\SGL}{{\bf GL}_{\rm seq}}
\newcommand{\SGLS}{{\bf GLS}_{\rm seq}}
\newcommand{\SGLGLS}{{\bf GL}({\bf S})_{\rm seq}}
\newcommand{\CFvdash}{\vdash_{\rm CutFree}}

\begin{document}

\title{Semantical cut-elimination for\\ the provability logic of true arithmetic}
\author{
Ryo Kashima\thanks{
Department of Mathematical and Computing Science, 
Tokyo Institute of technology.
{\tt kashima@is.titech.ac.jp}
}
\and Yutaka Kato\thanks{
Department of Mathematical and Computing Science, 
Tokyo Institute of technology.
}
}
\date{December 2022}
\maketitle

\begin{abstract}
The quasi-normal modal logic {\bf GLS} is a provability logic formalizing the arithmetical truth.
Kushida (2020) gave a sequent calculus for {\bf GLS} and proved the cut-elimination theorem.
This paper introduces semantical characterizations of {\bf GLS}
and gives a semantical proof of the cut-elimination theorem.
These characterizations can be generalized to other quasi-normal modal logics.
\end{abstract}

\section{Introduction}
\label{sec:intro}

The provability logic {\GLS} is an extension of the well-known provability logic {\GL}.
\footnote{
Following \cite{Boolos}, we use the name {\GLS},
while it is called {\bf S} and $G'$ in \cite{CZ,Solovay}.
Note that, in \cite{Gore}, {\bf GLS} is used as the name of sequent calculus for {\bf GL}.}
The axioms of {\GLS} are all theorems of {\GL} and all formulas $\Box\alpha \IMP \alpha$;
its sole inference rule is modus ponens. 
The importance of {\GLS} lies in the following two.
First, {\GLS} is complete with respect to the provability interpretation into the standard model of arithmetic.
See \cite{Boolos,Solovay} for details.
Second, {\GLS} is probably the best-known simple example 
of modal logics that are not normal but quasi-normal,
where quasi-normal modal logics are not required to be closed under the necessitation rule.
See \cite{CZ} for details of quasi-normal modal logics.

Recently a proof-theoretical study on {\GLS} was given by Kushida \cite{Kushida}.
He gave a natural sequent calculus by an ingenious way of using two kinds of sequents
and proved the cut-elimination theorem syntactically.

On the other hand, the following condition C0 has been known 
to be a semantical characterization of a formula $\varphi$ to be a theorem of {\GLS}.
\begin{itemize}
\item[C0.]
$\varphi$ is true at any $\SF{\varphi}$-reflexive world in any GL-model, 
where a GL-model is a transitive and converse well-founded Kripke-model,
and a world is said to be $\SF{\varphi}$-reflexive if the formula $\Box \alpha \IMP \alpha$
is true for any subformula $\Box\alpha$ of $\varphi$.
\end{itemize}
This paper presents further characterizations:
\begin{itemize}
\item[C1.]
There is a finite set $\SI$ of formulas such that $\varphi$ is true at any $\SI$-reflexive world in any GL-model,
where a world is said to be $\SI$-reflexive if the formula $\Box \alpha \IMP \alpha$
is true for any formula $\Box\alpha$ in $\SI$.

\item[C2.]
For any GL-model, there is a finite set $\SI$ of formulas such that $\varphi$ is true at any $\SI$-reflexive world.

\item[C3.]
For any infinitely descending sequence $w_1 R^{-1} w_2 R^{-1} w_3 \cdots$ in 
any GL-model, there is $i$ such that $\varphi$ is true at $w_j$ for all $j \geq i$.
($R$ is the accessibility relation of the GL-model.) 

\item[C4.]
For any infinitely descending sequence $w_1 R^{-1} w_2 R^{-1} w_3 \cdots$ in 
any GL-model, there is $i$ such that $\varphi$ is true at $w_i$.
\end{itemize}
Note that the implications 
`C0 $\Rightarrow$ C1 $\Rightarrow$ C2'
and
`C3 $\Rightarrow$ C4'
are trivial.
In Section~\ref{sec:Main}, we show the implications
`$({\GLS}\vdash\varphi) \Rightarrow$ C1',
`C2 $\Rightarrow$ C3',
and 
`C4 $\Rightarrow$ ($\varphi$ is cut-free provable in {\GLS})'.
These induce 
the soundness and completeness of {\GLS} with respect to the characterizations C1--C4,
and a simple semantical proof of the cut-elimination theorem.
Furthermore, we show a generalization of the characterizations to other quasi-normal modal logics.

\begin{remark}
\textup{
The characterization C0 is probably not explicitly stated in the literature, 
but the equivalence `$({\GLS}\vdash\varphi) \Longleftrightarrow$ C0'
is obvious because of the two well-known equivalences
`$({\GLS}\vdash\varphi)$
$\Longleftrightarrow$
$(\GL \vdash 
\bigwedge\{\Box\alpha\IMP\alpha \mid \Box\alpha \in \SF{\varphi}\} \IMP \varphi)$'
and
`$({\GL}\vdash\psi)$
$\Longleftrightarrow$
($\psi$ is valid in any GL-model)'.
}
\end{remark}

\begin{remark}
\textup{
`(${\GLS}\vdash\varphi) \Rightarrow$ C0',
namely the soundness of {\GLS} with respect to the characterization C0,
cannot be shown directly by induction,
whereas `(${\GLS}\vdash\varphi) \Rightarrow$ C1'
can be shown by straightforward induction on the proof of $\varphi$. 
So we cannot do semantical cut-elimination using the characterization C0.
}
\end{remark}

\begin{remark}
\textup{
The characterization C3 is essentially equivalent to Theorem~11.36 of Chagrov{\&}Zakharyaschev \cite{CZ},
which is described using the notion of general frames.
}
\end{remark}

\section{Definitions}

Formulas are constructed from 
propositional variables ($p, q, \ldots$),
propositional constant $\BOT$, 
logical operator $\IMP$,
and modal operator $\Box$.
The other operators are defined as abbreviations as usual.
The letters $\alpha, \beta, \varphi, \psi, \ldots$ denote formulas,
and $\GA, \DE, \PI, \SI, \ldots$ denote sets of formulas.
Parenthesis are omitted as, for example, 
$\Box\varphi \IMP \psi = (\Box\varphi) \IMP \psi$.
$\SF{\varphi}$ denotes the set of all the subformulas of $\varphi$. 
{\sf Fml} denotes the set of all formulas.

If $\GA$ and $\DE$ are finite sets of formulas, the expression $\SEQ{\GA}{\DE}$ is called a {\em sequent}.
As usual, for example, 
`$\SEQ{}{\GA, \varphi, \varphi, \DE}$'
denotes the sequent
$\SEQ{\emptyset}{\GA \cup \DE  \cup \{\varphi\}}.$
The expression $\Box\GA$ denotes the set $\{\Box\gamma \mid \gamma \in \GA\}$.

Sequent calculus ${\SGL}$, which has been well-studied (see \cite{Avron, Gore}), is defined as follows.
\begin{quote}
Initial sequents:
$\SEQ{\varphi}{\varphi}$
\mbox{ and }
$\SEQ{\BOT}{}$

Inference rules:
\[
\infer[\mbox{(cut)}]{\SEQ{\GA}{\DE}}{\SEQ{\GA}{\DE, \varphi} & \SEQ{\varphi, \GA}{\DE}}
\]
\[
\infer[\mbox{(weakening), where $\GA \subseteq \GA'$ and $\DE \subseteq \DE'$.}]{\SEQ{\GA'}{\DE'}}{\SEQ{\GA}{\DE}}
\]
\[
\infer[\mbox{($\IMP$L)}]{\SEQ{\varphi \IMP \psi, \GA}{\DE}}{
 \SEQ{\GA}{\DE, \varphi}
 &
 \SEQ{\psi, \GA}{\DE}
}
\quad
\infer[\mbox{($\IMP$R)}]{\SEQ{\GA}{\DE, \varphi \IMP \psi}}{
 \SEQ{\varphi, \GA}{\DE, \psi}
} 
\]
\[
\infer[\mbox{$(\Box_{\rm GL})$}]
{\SEQ{\Box\GA }{\Box\varphi}}{
  \SEQ{\GA, \Box\GA, \Box\varphi}{\varphi}
}
\]
\end{quote}
In other words, $\SGL$ is obtained from the sequent calculus {\bf LK} for classical propositional logic, by adding the rule $(\Box_{\rm GL})$.

Sequent calculus ${\SGLS}$, which is a slightly modified version of the calculus originated by Kushida \cite{Kushida}, is defined as follows.
${\SGLS}$ uses two kinds of sequents, {\em first level} and {\em second level}, 
while ${\SGL}$ uses only first level sequents.
Second level sequents are written using the symbol $\Rrightarrow$ instead of $\Rightarrow$.
${\SGLS}$ is defined by adding the following initial sequents and inference rules to ${\SGL}$.
\begin{quote}
Additional initial sequents:
$\SSEQ{\varphi}{\varphi}$
\mbox{ and }
$\SSEQ{\BOT}{}$

Additional inference rules:
\[
\infer[\mbox{(cut)}]{\SSEQ{\GA}{\DE}}{\SSEQ{\GA}{\DE, \varphi} & \SSEQ{\varphi, \GA}{\DE}}
\]
\[
\infer[\mbox{(weakening), where $\GA \subseteq \GA'$ and $\DE \subseteq \DE'$.}]{\SSEQ{\GA'}{\DE'}}{\SSEQ{\GA}{\DE}}
\]
\[
\infer[\mbox{($\IMP$L)}]{\SSEQ{\varphi \IMP \psi, \GA}{\DE}}{
 \SSEQ{\GA}{\DE, \varphi}
 &
 \SSEQ{\psi, \GA}{\DE}
}
\quad
\infer[\mbox{($\IMP$R)}]{\SSEQ{\GA}{\DE, \varphi \IMP \psi}}{
 \SSEQ{\varphi, \GA}{\DE, \psi}
} 
\]
\[
\infer[\mbox{($\Box$L), where 
$\Box\varphi$ is called the {\em principal formula} of this rule.}]
{\SSEQ{\Box\varphi, \GA }{\DE}}{
  \SSEQ{\varphi, \GA}{\DE}
}
\]
\[
\infer[\mbox{($\Rightarrow\Rrightarrow$)}]
{\SSEQ{\GA }{\DE}}{
  \SEQ{\GA}{\DE}
}
\]
\end{quote}
In other words, $\SGLS$ is obtained from {\bf LK} for both first and second level sequents, by adding the rules $(\Box_{\rm GL})$ on first level sequents, ($\Box$L) on second level sequents, and ($\Rightarrow\Rrightarrow$) which lifts the level of sequents.

We use the symbol $\blacktriangleright$ to denote $\Rightarrow$ or $\Rrightarrow$.
We write `$\SGLGLS \vdash \SSSEQ{\GA}{\DE}$' (or, `$\SGLGLS \CFvdash \SSSEQ{\GA}{\DE}$')
if the sequent $\SSSEQ{\GA}{\DE}$ is provable (or, provable without using the rule (cut), respectively)
in $\SGLGLS$.

$\SGLS$ is a conservative extension of $\SGL$; 
that is, provability of first level sequents are equivalent between two calculus.
This comes from the one-wayness of the rule ($\Rightarrow\Rrightarrow$).
Thus we have the following.
\begin{quote}
$
(\SGLS \vdash \SEQ{}{\varphi})
\Longleftrightarrow
(\SGL \vdash \SEQ{}{\varphi})
\Longleftrightarrow
\mbox{($\varphi$ is a theorem of {\GL})}.
$

$
(\SGLS \vdash \SSEQ{}{\varphi})
\Longleftrightarrow 
\mbox{($\varphi$ is a theorem of {\GLS})}.
$
\end{quote}

The cut-elimination theorem holds for these calculi;
see 
\cite{Gore} for syntactical proof for ${\SGL}$,
\cite{Avron} for semantical proof for ${\SGL}$,
and 
\cite{Kushida} for syntactical proof for ${\SGLS}$.
Semantical proof for ${\SGLS}$ is given by this paper.

\medskip

By {\em GL-model}, we mean a transitive and converse well-founded Kripke model.
That is,  $\langle W, R, V\rangle$ is a GL-model
if 
$W$ is a non-empty set of worlds,
$R \subseteq W \times W$ is transitive,
there is no infinitely ascending sequence $x_1 R x_2 R x_3 \cdots$,
and 
$V: W \times {\sf Fml}  \to \{\T, \F\}$ is a valuation
that satisfies the following.
$V(w,\BOT) = \F$.
$V(w, \varphi\IMP\psi) = \T
\Longleftrightarrow
V(w, \varphi) = \F
\mbox{ or }
V(w, \psi) = \T.$
$V(w,\Box\varphi) = \T
\Longleftrightarrow
(\forall w')(wRw' \Rightarrow  V(w',\varphi) = \T).$
Truth of a sequent $\SSSEQ{\GA}{\DE}$ is defined by truth of the formula $\bigwedge\GA \IMP \bigvee \DE$;
that is, 
$V(w, (\SSSEQ{\GA}{\DE})) = \T
\Longleftrightarrow
(\exists \gamma \in \GA)(V(w,\gamma) = \F)$ or $(\exists \delta \in \DE)(V(w,\delta)= \T)$.

Let $\SI$ be a set of formulas.
We say that a world $w$ is {\em $\SI$-reflexive} if and only if for any $\Box\alpha \in \SI$, $V(w, \Box\alpha \IMP \alpha) = \T$.

\section{Results}
\label{sec:Main}

\begin{theorem}
\label{Th:Main}
For any second level sequent $\SSEQ{\PS}{\PH}$, the following six conditions are equivalent.
\begin{enumerate}
\item\label{item:A}
There is a finite set $\SI$ of formulas such that for any GL-model $\langle W, R, V\rangle$  and any $\SI$-reflexive world $w \in W$, $V(w, (\SSEQ{\PS}{\PH})) = \T$.

\item\label{item:B}
For any GL-model $\langle W, R, V\rangle$, there is a finite set $\SI$ of formulas such that for any $\SI$-reflexive world $w \in W$, $V(w, (\SSEQ{\PS}{\PH})) = \T$.

\item\label{item:C}
For any GL-model $\langle W, R, V\rangle$ and any infinitely descending sequence \\
$w_1 R^{-1} w_2 R^{-1} w_3 \cdots$ in $W$, there is a number $i$ such that for any $j\geq i$, $V(w_j, (\SSEQ{\PS}{\PH})) = \T$.

\item\label{item:D}
For any GL-model $\langle W, R, V\rangle$ and any infinitely descending sequence \\
$w_1 R^{-1} w_2 R^{-1} w_3 \cdots$ in $W$, there is a number $i$ such that $V(w_i, (\SSEQ{\PS}{\PH})) = \T$.

\item\label{item:E}
$\SGLS \CFvdash \SSEQ{\PS}{\PH}$.

\item\label{item:F}
$\SGLS \vdash \SSEQ{\PS}{\PH}$.
\end{enumerate}
\end{theorem}

Note that 
the implications 
`$\ref{item:A}\Rightarrow\ref{item:B}$', 
`$\ref{item:C}\Rightarrow\ref{item:D}$', 
and 
`$\ref{item:E}\Rightarrow\ref{item:F}$' 
are trivial.
In the following, we show 
`$\ref{item:F}\Rightarrow\ref{item:A}$', 
`$\ref{item:B}\Rightarrow\ref{item:C}$', 
and 
`$\ref{item:D}\Rightarrow\ref{item:E}$'.

\begin{proof}[Proof of ` $\ref{item:F} \Rightarrow \ref{item:A}$' (Soundness of $\SGLS$)]
Suppose the condition \ref{item:F} holds;
that is, there is a proof ${\cal P}$ of $\SSEQ{\PS}{\PH}$ in $\SGLS$.
We define $\SI$ to be the set of principal formulas of all ($\Box$L) rules in ${\cal P}$.
Then, for any subproof ${\cal P}'$ of ${\cal P}$, we can show the following by induction on the size of ${\cal P'}$:
{\em 
If the conclusion of ${\cal P'}$ is $\SEQ{\GA}{\DE}$, then $V(w, (\SEQ{\GA}{\DE}))=\T$
for any GL-model $\langle W, R, V\rangle$ and any world $w$.
If the conclusion of ${\cal P'}$ is $\SSEQ{\GA}{\DE}$, then $V(w, (\SSEQ{\GA}{\DE}))=\T$
for any GL-model $\langle W, R, V\rangle$ and any $\SI$-reflexive world $w$.}
\end{proof}

\begin{lemma}
\label{lm:reflexive}
For any formula $\alpha$, any GL-model $\langle W, R, V\rangle$, and any infinitely descending sequence $w_1 R^{-1} w_2 R^{-1} w_3 \cdots$ in $W$, there is a number $i$ such that for any $j\geq i$, $V(w_j, \Box\alpha\IMP\alpha) = \T$.
\end{lemma}
\begin{proof}
If $V(w_n, \alpha)=\T$ for all $n$, then $V(w_j, \Box\alpha\IMP\alpha)=\T$ for all $j \geq 1$.
If $V(w_n, \alpha)=\F$ for some $n$, then for any $j\geq (n+1)$, $V(w_j, \Box\alpha)=\F$
(therefore $V(w_j, \Box\alpha\IMP\alpha)=\T$) because $w_j R w_n$.
\end{proof}

\begin{proof}[Proof of ` $\ref{item:B} \Rightarrow \ref{item:C}$']
Given any GL-model $\langle W, R, V\rangle$ and any infinitely descending sequence $w_1 R^{-1} w_2 R^{-1} w_3 \cdots$, we get the finite set $\SI$ by the condition \ref{item:B}.
Then, by Lemma~\ref{lm:reflexive}, there is a sufficiently large $n$ such that $w_j$ is $\SI$-reflexive for any $j\geq n$.
Therefore, by the condition \ref{item:B}, we have $V(w_j, (\SSEQ{\PS}{\PH}))=\T$ for any $j\geq n$.
\end{proof}

For the proof of `$\ref{item:D} \Rightarrow \ref{item:E}$', we arbitrarily fix a sequent $\SSEQ{\PS}{\PH}$.
Then we will henceforth consider only sequents consisting of subformulas of $\PS, \PH$.
In other words, {\em when we write a sequent $\SSSEQ{\GA}{\DE}$, the condition $\GA\cup\DE \subseteq \SF{\PS,\PH}$ is automatically assumed from now on.}

We define conditions on a sequent $\SSSEQ{\GA}{\DE}$.
\begin{quote}
(${\IMP}$L) \ 
If $\varphi\IMP\psi \in \GA$, then $\varphi \in \DE$ or $\psi \in \GA$.
\\
(${\IMP}$R) \ 
If $\varphi\IMP\psi \in \DE$, then $\varphi \in \GA$ and $\psi \in \DE$.
\\
(${\Box}$L) \ 
If $\Box\varphi \in \GA$, then $\varphi \in \GA$.
\end{quote}
We say that a first level sequent $\SEQ{\GA}{\DE}$ is {\em saturated} if and only if the conditions 
(${\IMP}$L) and (${\IMP}$R) hold for any formulas $\varphi$ and $\psi$.
We say that a second level sequent $\SSEQ{\GA}{\DE}$ is {\em saturated} if and only if all the above conditions hold for any formulas $\varphi$ and $\psi$.

The following lemma is a standard tool for semantical cut-elimination.

\begin{lemma}
\label{lm:saturation}
If $\SGLS \not\CFvdash \SSSEQ{\GA}{\DE}$, then there is a saturated sequent $\SSSEQ{\GA^+}{\DE^+}$ such that $\GA \subseteq \GA^+$, $\DE \subseteq \DE^+$, and $\SGLS \not\CFvdash \SSSEQ{\GA^+}{\DE^+}$.
\end{lemma}
\begin{proof}
$\SSSEQ{\GA^+}{\DE^+}$ can be obtained by appropriately adding formulas to $\SSSEQ{\GA}{\DE}$ while preserving cut-free unprovability.
\end{proof}

\begin{proof}[Proof of ` $\ref{item:D} \Rightarrow \ref{item:E}$' (Cut-free Completeness of $\SGLS$)]
Suppose $\SGLS \not\CFvdash \SSEQ{\PS}{\PH}$.
We show that  there is a GL-model $\langle W, R, V\rangle$ and an infinitely descending sequence $w_1 R^{-1} w_2 R^{-1} w_3 \cdots$ such that $V(w_i, (\SSEQ{\PS}{\PH}))=\F$ for any $i$.

First, we follow the well-known proof of cut-free completeness of $\SGL$ by Avron \cite{Avron}.
We define $\langle W_0, R_0, V_0\rangle$ as below.
\begin{quote}
$W_0 = 
\{\SEQ{\GA}{\DE} \mid
\SGLS\not\CFvdash\SEQ{\GA}{\DE} \mbox{ and $\SEQ{\GA}{\DE}$ is saturated}\}.$

$(\SEQ{\GA}{\DE})R_0(\SEQ{\GA'}{\DE'})$
$\Longleftrightarrow$
$\GA_{\Box} \subsetneq \GA'_{\Box}$ and $\GA_{\Box} \subseteq \GA'$, 
where for a set $\TA$, we define $\TA_{\Box} = \{\varphi \mid \Box\varphi \in \TA\}$.

$V_0((\SEQ{\GA}{\DE}), p) = \T$
$\Longleftrightarrow$
$p \in \GA$.
\end{quote}
We can show that $R_0$ is transitive and converse well-founded.
Moreover, the following  hold for any formula $\varphi$ and any world $(\SEQ{\GA}{\DE}) \in W_0$.
\begin{equation}
\mbox{If $\varphi \in \GA$, then $V_0((\SEQ{\GA}{\DE}), \varphi) = \T$.}
\label{eq:truth}
\end{equation}
\begin{equation}
\mbox{If $\varphi \in \DE$, then $V_0((\SEQ{\GA}{\DE}), \varphi) = \F$.}
\label{eq:false}
\end{equation}
These are simultaneously proved by induction on $\varphi$.
When $\varphi = \Box\psi \in \DE$, we apply Lemma~\ref{lm:saturation} 
to $\SEQ{\GA_{\Box}, \Box(\GA_{\Box}), \Box\psi}{\psi}$
in order to show the existence of $\SEQ{\GA'}{\DE'}$ such that $(\SEQ{\GA}{\DE})R_0(\SEQ{\GA'}{\DE'})$ and 
$\psi \in \DE'$.

Next, we apply Lemma~\ref{lm:saturation} to $\SSEQ{\PS}{\PH}$, and we get a saturated sequent $\SSEQ{\PS^+}{\PH^+}$ which is not cut-free provable.
Then, also the first level sequent $\SEQ{\PS^+}{\PH^+}$ is not cut-free provable because of the inference rule $(\Rightarrow\Rrightarrow)$; thus it is an element of $W_0$.

Now the GL-model $\langle W, R, V\rangle$ we want is defined as follows.
\begin{quote}
$W = W_0 \cup {\mathbb N}$, where ${\mathbb N} = \{1, 2, 3, \ldots\}$.

$xRy \Longleftrightarrow 
(x,y \in W_0 \mbox{ and }  xR_0y)
\mbox{ or }
(x \in {\mathbb N} \mbox{ and } y = (\SEQ{\PS^+}{\PH^+}))
\mbox{ or }
(x \in {\mathbb N} \mbox{ and } (\SEQ{\PS^+}{\PH^+})R_0y)
\mbox{ or }
(x,y \in {\mathbb N} \mbox{ and } x>y).
$

$V((\SEQ{\GA}{\DE}), p) = V_0((\SEQ{\GA}{\DE}), p)$ and $V(n,p) = V_0((\SEQ{\PS^+}{\PH^+}), p)$.
\end{quote}
It is easy to show that $R$ is transitive and converse well-founded.
Moreover, 
for any world $(\SEQ{\GA}{\DE}) \in W_0$ and any formula $\varphi$, 
we have $V((\SEQ{\GA}{\DE}), \varphi) = V_0((\SEQ{\GA}{\DE}), \varphi)$ 
because the difference between two relations $R$ and $R_0$ occurs only in $(n, -)$;
thus the properties (\ref{eq:truth}) and (\ref{eq:false}) hold also for $V$.

We show the following for any $n \in {\mathbb N}$ and any formula $\varphi$.
\begin{equation*}
\mbox{If $\varphi \in \PS^+$, then $V(n, \varphi) = \T$.}
\label{eq:truth-GLS}
\end{equation*}
\begin{equation*}
\mbox{If $\varphi \in \PH^+$, then $V(n, \varphi) = \F$.}
\label{eq:false-GLS}
\end{equation*}
These are simultaneously proved by induction on $\varphi$.
When $\varphi = \Box\psi \in \PS^+$, we use the condition (${\Box}$L) of the saturated second level sequent $\SSEQ{\PS^+}{\PH^+}$
and the property (\ref{eq:truth}) for $\SEQ{\PS^+}{\PH^+}$.

Consequently, $1 R^{-1} 2 R^{-1} 3 R^{-1} \cdots$ is an infinitely descending sequence such that 
$V(n, (\SSEQ{\PS}{\PH})) = \F$ for all $n$.
\end{proof}

\medskip

Finally, we show a generalization of our characterizations.

Let ${\cal F}$ be a class of Kripke frames.
We say that a model $\langle W, R, V\rangle$ is an {\em ${\cal F}$-model}
if $\langle W, R\rangle \in {\cal F}$.
In the statements C1--C4 in Section~\ref{sec:intro},
reading `GL-model' as `${\cal F}$-model' and `$w_1 R^{-1} w_2 R^{-1} w_3 \cdots$' as `$w_n R w_m$ for all $n > m$'
yields four characterizations, say C1$^{\cal F}$--C4$^{\cal F}$.
(Note that $R$ may not be transitive; 
so `$w_1 R^{-1} w_2 R^{-1} w_3 \cdots$' does not imply `$w_n R w_m$ for all $n > m$'.)
Let $L$ be the normal modal logic that is 
sound and complete with respect to ${\cal F}$, 
and $L^+$ be the quasi-normal modal logic 
that is the closure of $L \cup \{\Box\alpha\IMP\alpha \mid \alpha \in {\sf Fml}\}$ under modus ponens.

\begin{theorem}
If ${\cal F}$ is closed under the construction of frames 
that makes $\langle W_0, R_0\rangle$ into $\langle W, R\rangle$
in the proof of `$\ref{item:D} \Rightarrow \ref{item:E}$' of Theorem~\ref{Th:Main},
then
$L^+$ is sound and complete with respect to the characterizations C1$^{\cal F}$--C4$^{\cal F}$.
\end{theorem}
\begin{proof}[Proof (sketch).]
Similar to Theorem~\ref{Th:Main}.
In the above proof, we use the GL-model 
$\langle W_0, R_0, V_0\rangle$
and the world $\SEQ{\PS^+}{\PH^+}$ in $W_0$.
Here instead, we use 
an ${\cal F}$-model $\langle W_0^{L}, R_0^{L}, V_0^{L}\rangle$
and a world {\sf w}
which are obtained by the following fact.
\begin{align*}
L^+ \not\vdash \varphi
&\Longrightarrow
L \not\vdash
\bigwedge\{\Box\alpha\IMP\alpha \mid \Box\alpha \in \SF{\varphi}\} \IMP \varphi
\\
&\Longrightarrow
\mbox{there is an ${\cal F}$-model $\langle W_0^{L}, R_0^{L}, V_0^{L}\rangle$ and a world {\sf w} such that}
\\
&\qquad
V_0^{L}({\sf w}, \Box\alpha\IMP\alpha) = \T
\mbox{ for any } \Box\alpha \in \SF{\varphi}
\mbox{ and }
V_0^{L}({\sf w}, \varphi) = \F.
\end{align*}
Then we construct an ${\cal F}$-model $\langle W_0^{L}\cup{\mathbb N}, R^{L}, V^{L}\rangle$, 
and we show 
$V^{L}(n, \psi) = V_0^{L}({\sf w}, \psi)$
for any $n \in {\mathbb N}$ and any $\psi \in \SF{\varphi}$.
\end{proof}

\end{document}